\documentclass[12pt,twoside,reqno,psamsfonts]{amsart}

\usepackage[OT1]{fontenc}
\usepackage{type1cm}
\usepackage{enumerate}
\usepackage{amssymb}
\usepackage{mathrsfs}
\usepackage[dvips]{graphicx}
\usepackage{psfrag}          
\usepackage{geometry}        
\usepackage{version}         
\usepackage{xcolor}

\numberwithin{equation}{section}

\theoremstyle{plain}
\newtheorem{theorem}{Theorem}[section]

\newtheorem{proposition}[theorem]{Proposition}

\theoremstyle{remark}
\newtheorem{remark}[theorem]{Remark}

\newtheorem{example}[theorem]{Example}

\theoremstyle{definition}
\newtheorem{definition}[theorem]{Definition}
\newtheorem{question}[theorem]{Question}

\newcommand{\LL}{{\mathscr L}}
\newcommand{\R}{\mathbb{R}}
\newcommand{\Q}{\mathbb{Q}}
\newcommand{\N}{\mathbb{N}}

\DeclareMathOperator{\diam}{diam}
\DeclareMathOperator{\Geo}{Geo}
\DeclareMathOperator{\GeoOpt}{GeoOpt}

\begin{document}

\title[Slopes and optimal maps]
{Slopes of Kantorovich potentials and existence of optimal transport
maps in \\ metric measure spaces}
\author{Luigi Ambrosio}
\author{Tapio Rajala}
\address{Scuola Normale Superiore\\
Piazza dei Cavalieri 7\\
I-56127 Pisa\\ Italy}
\email{luigi.ambrosio@sns.it}
\email{tapio.rajala@sns.it}

\date{\today}


\begin{abstract}
 We study optimal transportation with the quadratic cost function in geodesic metric spaces
 satisfying suitable non-branching assumptions. We introduce and
 study the notions of slope along curves and along geodesics and we apply
 the latter to prove suitable generalizations of Brenier's theorem of existence
 of optimal maps.
\end{abstract}


\maketitle


 \section{Introduction}

 The problem of finding an optimal way to transport mass has a long
 history, starting from Monge's seminal paper \cite{M1781}. The optimality of a transport
can be measured in many ways, depending on the choice of the cost
function. In this paper we focus on the case when the cost is the
square of the distance.

 Given two positive and finite measures $\mu$ and $\nu$ on some metric space $(X,d)$ with the same total mass,
 which we may normalize to 1, our task is then to study whether the infimum
 \begin{equation}\label{eq:Monge}
  \inf \int_Xd^2(x,T(x))d\mu(x),
 \end{equation}
 over all possible $\mu$-measurable maps $T \colon X \to X$ which send the measure $\mu$ to $\nu$, is attained.
 If such a minimizing map exists, we call it an optimal transport map between $\mu$ and
 $\nu$. Existence of optimal maps or even of admissible ones is
 problematic, for instance no admissible map exists when $\mu$ is a
 Dirac mass and $\nu$ is not a Dirac mass.

 Kantorovich's
 relaxed \cite{K1942, K1948} formulation of the optimal transport problem consists in finding the infimum
 \begin{equation}\label{eq:Kantorovich}
  \inf \int_{X\times X}d^2(x,y)\pi(x,y)
 \end{equation}
 over all possible transport plans, i.e. probability measures $\pi$ on $X\times X$ which have $\mu$ and $\nu$
 as marginals.
 Again, if there is a measure which attains the infimum, it is called an optimal transport plan between $\mu$ and $\nu$.
 Notice that transport plans can split measure, and so they avoid the problem faced by transport maps. In fact, not only
 the Kantorovich formulation of the problem is well-posed, but the
 infimum is attained (possibly infinite) under the only assumption
 that $(X,d)$ is complete and separable. Since we will be dealing
 with geodesic metric spaces, we will mostly work with the equivalent formulation in terms
 of geodesic transport plans, i.e. probability measures in the space $\Geo(X)$
 of constant speed geodesics parameterized on $[0,1]$, with marginal
 conditions at $t=0$ and $t=1$.

 In general, it seems to be a difficult problem to find necessary and sufficient conditions under which Monge's
 problem has a solution, but by now several sufficient conditions are known. For the quadratic cost
 in the Euclidean setting it was proved independently by
 Brenier \cite{B1991} and Smith and Knott \cite{SK1987} that there exists a unique optimal map $T$,
 given by the gradient of a convex function, provided that
 $\mu$ is absolutely continuous with respect to the Lebesgue measure. This
result was generalized to Riemannian manifolds by
 McCann \cite{M2001}, to Alexandrov spaces by Bertrand \cite{B2008}
 to the Heisenberg group by Ambrosio and Rigot \cite{AR2008} and,
 very recently, to non-branching metric spaces with Ricci curvature
 bounded from below (in the sense of Lott, Sturm and Villani) by Gigli
 \cite{G2011}. Notice that in all these results a reference measure
 $m$ (Lebesgue measure, Riemannian volume, Haar measure, etc.) plays a role,
 so the proper setting for this question is the
 family of metric measure spaces $(X,d,m)$.

 In another recent paper \cite{AGS2011}, a \emph{metric} Brenier
 theorem is proved under mild assumptions on $(X,d,m)$, see Theorem~10.3 and Remark~10.7
 therein. In the case
 when $(X,d)$ has bounded diameter and $m$ is a finite measure, the
 main assumption is the existence of bounds on the relative entropy along geodesics
 (a condition weaker than the $CD(K,\infty)$ condition of Lott,
 Sturm and Villani) and the metric Brenier theorem states that, for
 any optimal geodesic plans $\pi$, it holds
 \begin{equation}\label{eq:metbre}
 |\nabla^+\varphi|(\gamma_0)=d(\gamma_0,\gamma_1)\qquad\text{$\pi$-a.e. in $\Geo(X)$}
 \end{equation}
 (here $\varphi$ is any Kantorovich potential and $|\nabla^+\varphi|$
 is its ascending slope).
 In other words, the transportation distance depends $\mu$-a.e. only
 on the initial point.

 This result raises some questions that we plan to investigate in
 this paper: the first one is to understand under which additional
 assumptions one can really recover an optimal map, the second one
 is about the differentiability of $\varphi$ along geodesics used by
 the optimal plan.

 In connection with the first question we start from this heuristic
 idea more or less implicit in many proofs: under appropriate
 structural assumptions on the space,
 \eqref{eq:metbre} identifies the ``initial velocity'' of the
 geodesic. Indeed, assuming suitable non-branching assumptions on the space
 and on its tangent metric spaces we can perform a suitable blow-up
 analysis that leads to the existence of optimal maps. The proof of this
 result requires a detailed analysis of the proof of the metric Brenier
 theorem in \cite{AGS2011} and the introduction of a sharper notion
 of ascending slope, namely the ascending slope $|\nabla^+_g\varphi|$
 along geodesics. Since we believe that this concept has an independent
 interest we compare this slope to the usual one and to the
 slope along curves, we provide an example and we raise some open problems.
 Coming back to the existence of optimal maps, our result covers as a particular
 case the Euclidean, the Riemannian and the Alexandrov case, see also the paragraph
 immediately after Theorem~\ref{thm:BrenierNonBranch} for a more detailed discussion.

 In connection with the second question, it has already been proved in
 Theorem~10.4 of \cite{AGS2011} a ``differentiability in mean'', namely
 \begin{equation}\label{eq:inmean}
 \lim_{t\downarrow
 0}\frac{\varphi(\gamma_0)-\varphi(\gamma_t)}{d(\gamma_0,\gamma_t)}=|\nabla^+\varphi|(\gamma_0)
 \qquad\text{in $L^2(\pi)$.}
 \end{equation}
 This weak differentiability property plays an important role in the
 subsequent paper \cite{AGS2011b}, for the computation of the
 derivative of the entropy along geodesics.
 Here, under an additional doubling assumption on $m$, we are able
 to improve \eqref{eq:inmean} to a pointwise differentiability
 property, so that
 $$
 \varphi(\gamma_t)=\varphi(\gamma_0)-t|\nabla^+\varphi|^2(\gamma_0)+o(t)
 $$
 for $\pi$-a.e. $\gamma\in\Geo(X)$.

\noindent {\bf Acknowledgement.} The authors acknowledge the support
of the ERC ADG GeMeThNES. The second author also acknowledges the
support of the Academy of Finland, project no. 137528.


\section{Non-branching metric spaces}\label{sec:nonbranch}

Let us start by laying out the definitions for the metric spaces
that will be used in this paper. First of all, we will be working
exclusively in metric spaces $(X,d)$ which are complete and
separable. Second, by measure in $(X,d)$ we mean a nonnegative Borel
measure, finite on bounded sets. We will mainly consider
metric spaces $X$ equipped with a \emph{doubling measure} $m$ meaning that
there exists a constant $0 < C < \infty$ so that for all $0 < r < \diam(X)$
and $x \in X$ we have
\[
 m(B(x,2r)) \le Cm(B(x,r)).
\]
A related notion for metric spaces where the measure has not been
specified is that of a \emph{doubling metric space}, which means
that there exists an integer $N\geq 1$ so that, for all $0 < r <
\infty$, any ball of radius $2r$ can be covered by $N$ balls of
radius $r$. It is obvious that if there exists a doubling measure on
$X$ then the space $X$ has to be doubling as well. The converse is
also true for complete metric space, see for example \cite{LS1998}
and \cite{KRS2010}.

 We call any absolutely continuous map $\gamma \colon [a,b] \to X$ a curve and use the abbreviation $\gamma_s = \gamma(s)$.
 The length of the curve $\gamma$ is defined as
 \[
  l(\gamma) = \sup \left\{\sum_{i=1}^Nd(\gamma_{t_i},\gamma_{t_{i-1}}) ~:~
  a \le t_0 < t_1 < \cdots < t_N \le b, N \in \N\right\}.
 \]
 We call the curve $\gamma \colon [a,b] \to X$ a \emph{geodesic} if $l(\gamma) = d(\gamma_a,\gamma_b)$.
 The metric space $X$ itself is called \emph{geodesic} if any two points $x,\,y\in X$ can
 be connected with a geodesic, i.e. there exists a geodesic
 $\gamma \colon [a,b] \to X$ with $\gamma_a=x$ and $\gamma_b=y$. Sometimes, when there is no
 danger of confusion, we also call the image of a geodesic a geodesic.

 The speed of a curve $\gamma$ is given by
 \[
  |\dot\gamma|(t) = \lim_{s \to t}\frac{d(\gamma_s,\gamma_t)}{|s-t|}
 \]
 whenever the limit exists. It is not hard to prove, see for instance
 Theorem~1.1.2 in \cite{AGS2008},
  that it indeed exists at $\LL^1$-almost every point $t \in [a,b]$,
 where $\LL^1$ is the Lebesgue measure on $\R$, and that
 $l(\gamma)=\int_a^b|\dot\gamma|(t)dt$.

 We denote by $\Geo(X)$ the set of all constant speed geodesics in $X$ which are parametrized by $[0,1]$, namely
 $d(\gamma_s,\gamma_t)=|t-s|d(\gamma_0,\gamma_1)$ for all $s,\,t\in [0,1]$. By a reparameterization argument, constant
 speed geodesics connecting any two given points exist in any geodesic space. We equip the space $\Geo(X)$
 with the distance
 \[
  d^*(\gamma,\tilde\gamma) = \max_{t \in [0,1]}d(\gamma_t,\tilde\gamma_t)
 \]
 and note that $(\Geo(X),d^*)$ is also complete and separable since the underlying metric space is.
 We will also use the convenient notation of \emph{evaluation map} $e_t \colon \Geo(X) \to
 X$, defined as $e_t(\gamma) = \gamma_t$ for all $t \in [0,1]$.

 With the basic notation related to geodesics now fixed we are ready to introduce the two definitions of
 non-branching which play a crucial role in our results.

 \begin{definition}
  We call a geodesic metric space $(X,d)$ \emph{non-branching} if for any two constant speed geodesics
  $\gamma,\,\gamma' \colon [0,1] \to X$ with $\gamma_0 = \gamma'_0$ and $\gamma_s = \gamma'_s$ for some $s\in (0,1)$
  we have $\gamma_t = \gamma'_t$ for all $t \in [0,1]$.
 \end{definition}

 We would like
 to use non-branching on the level of the tangent spaces. However, two distinct geodesics of
 a metric space can collapse into a single geodesic of the tangent space in the blow-up. To control such collapsing we
 will assume a stronger version of non-branching.

 \begin{definition}
  We call a geodesic metric space $(X,d)$ \emph{strongly non-branching} if for any two constant speed geodesics
  $\gamma,\,\gamma' \colon [0,1] \to X$ with $\gamma_0 = \gamma'_0$, $\gamma_1 \ne \gamma'_1$ and
  $d(\gamma_0,\gamma_1) > 0$, we have
  \begin{equation}\label{eq:strongnonbrancdef}
   \liminf_{t\downarrow 0} \frac{d(\gamma_t,\gamma'_t)}{d(\gamma_0,\gamma_t)} > 0.
  \end{equation}
 \end{definition}

 In our main theorem, Theorem~\ref{thm:BrenierNonBranch}, we assume that the space is strongly non-branching
 and that at almost every point we have some non-branching tangent space.
 Before defining what we mean by tangent space we recall the definitions of Hausdorff- and
 Gromov-Hausdorff-distance. The \emph{Hausdorff-distance} between closed sets $A,\,B\subset X$ is defined as
 \[
  d_H(A,B) = \max\left\{\sup_{a\in A}\inf_{b\in B}d(a,b),\sup_{b\in B}\inf_{a\in A}d(a,b)\right\}.
 \]
 Using the Hausdorff-distance, the \emph{Gromov-Hausdorff-distance} between two metric spaces $(X,d_X)$ and $(Y,d_Y)$
 is then defined as
 \[
 d_{GH}(X,Y) = \inf d_H(f(X),g(Y)),
 \]
 where the infimum is taken over all metric spaces $(Z,d_Z)$ and isometries $f \colon X \to Z$, $g \colon Y \to Z$.
 Finally, a sequence $(X_n,d_n,x_n)_{n=1}^\infty$ of metric spaces $(X_n,d_n)$ and points $x_n\in X_n$ is said
 to converge to $(X,d,x)$ in the \emph{pointed Gromov-Hausdorff} sense if
 \[
  \lim_{n \to \infty} d_{GH}(\overline{B}_{X_n}(x_n,r),\overline{B}_X(x,r)) =
  0\qquad\forall r>0,
 \]
 where by $\overline{B}$ we denote the closed ball.
 Given a metric space $(X,d)$ and a scaling factor $r>0$, we define a rescaled metric $d_r$ on $X$ by setting
  \[
   d_r(x,y) = \frac1rd(x,y)
  \]
  for all $x,\,y\in X$.

 \begin{definition}
  Let $(X,d)$ be a metric space. We call a metric space $(Y,\rho)$ \emph{tangent} to $(X,d)$ at $x\in X$ if
  there exist a sequence $(r_n)\downarrow 0$ and $y\in Y$ so that
  \[
    (X,d_{r_n},x) \underset{n \to \infty}{\longrightarrow} (Y,\rho,y)
  \]
  in the pointed Gromov-Hausdorff convergence.
 \end{definition}

 Notice that our definition of a tangent space is weaker than Gromov's original notion. He required
 the tangent space to be the full limit of the spaces $(X,d_r,x)$ as $r\downarrow 0$, whereas in our definition
 we only require convergence along a subsequence. As a consequence, with our definition the space
 $(X,d)$ can in principle have a huge collection of different tangent spaces at a single point.

 We will use the following well-known result (see for instance \cite[Proposition~2.7]{AGS2011b})
 which allows us to move from the Gromov-Hausdorff convergence to Hausdorff-convergence.

 \begin{theorem}\label{thm:Gromov}
  If $(X_n,d_n) \to (X,d)$ in the Gromov-Hausdorff convergence
  then there exist a space $(Z,d_Z)$ and isometric embeddings
  $$
  i_n \colon (X_n,d_n) \to (Z,d_Z),\qquad i \colon (X,d) \to (Z,d_Z)
  $$ so that $i_n(X_n) \to i(X)$ in the Hausdorff convergence.\\
  In addition, if $(X_n,d_n)$ are equi-compact, then
  $(Z,d_Z)$ can be taken to be a compact metric space.
 \end{theorem}

%

\section{Gradients along geodesics}\label{sec:gradients}

 Let us recall some basic definitions in measure theory.
 The collection of \emph{universally measurable sets} of the space $X$, denoted by $\mathscr{B}^*(X)$,
 is the $\sigma$-algebra of the sets which are $\mu$-measurable for all finite nonnegative Borel measure
 $\mu$ of $(X,d)$. The collection of all Borel sets of $(X,d)$ will be denoted by $\mathscr{B}(X)$.

 Now we turn to our next set of definitions that concern metric differentials.

 \begin{definition}
  Given a function $f \colon X \to \R$ we define the \emph{lower ascending slope along geodesics} of $f$ at
  $x\in X$ as
  \[
   |\nabla_g^+f|(x) = \sup_{\gamma}\liminf_{s \downarrow 0} \frac{[f(\gamma_s)-f(x)]^+}{d(\gamma_s,x)},
  \]
  where $^+$ denotes the positive part and
  the supremum is taken over all nonconstant geodesics in $X$ that start from the point $x$.
 \end{definition}

Here \emph{ascending} refers to the fact that we are taking the
positive part of the difference quotient and \emph{lower} refers to
the fact that we are taking the $\liminf$, rather than the
$\limsup$.

 \begin{proposition}\label{prop:measurability}
  Suppose that $f \colon X \to \R$ is continuous. Then $|\nabla_g^+f|$ is universally measurable.
 \end{proposition}
 \begin{proof}
  Let $T\in \R$ and consider the sublevel set $\{|\nabla_g^+f| > T\} \subset X$. For the universal measurability
  it is sufficient to show that this set is Suslin.
  Because $X$ is complete and separable, so are $\Geo(X)$ and $X \times \Geo(X)$. Therefore $\{|\nabla_g^+f| > T\}$,
  being the projection of the set
  \[
   \left\{(x,\gamma) ~:~ \gamma_0 = x,\,\liminf_{s \downarrow 0}
   \frac{[f(\gamma_s)-f(x)]^+}{d(\gamma_s,x)} > T \right\} \subset X \times \Geo(X)
  \]
  to the space $X$, is indeed Suslin since the projected set can be written as
  \[
   \bigcap_{t\in \Q \cap (0,1)} \bigcup_{s \in \Q \cap (0,t)}
   \left\{(x,\gamma) ~:~ \gamma_0 = x,\frac{[f(\gamma_s)-f(x)]^+}{d(\gamma_s,x)} > T \right\}
  \]
  and so it is Borel (countable intersection of countable unions of open sets).
 \end{proof}

 \begin{definition}
  A function $g \colon X \to [0,\infty]$ is an \emph{upper gradient along geodesics} of a
  function $f \colon X \to \R$ if for any $\gamma \in \Geo(X)$ we have
  \begin{equation}\label{eq:sugdefinition}
   |f(\gamma_0)-f(\gamma_1)| \le \int_\gamma g,
  \end{equation}
  where the integral along $\gamma$ is understood as
  \[
   \int_\gamma g = l(\gamma)\int_0^1g(\gamma_s)ds.
  \]
 \end{definition}

 The almost everywhere differentiability of Lipschitz functions on the real line implies that ascending slopes are
 upper gradients for Lipschitz functions. We include the easy proof of this fact here for the convenience of the reader.
 Recall that a function $f \colon X \to \R$ is called Lipschitz if there exists a constant $0\leq L < \infty$ so that for any
 two points $x,\,y \in X$ we have
 \[
  |f(x)-f(y)|\le Ld(x,y).
 \]

 \begin{proposition}\label{prop:lipgrad}
  Let $f \colon X \to \R$ be Lipschitz. Then the lower ascending slope along geodesics is an upper
  gradient along geodesics.
 \end{proposition}

 \begin{proof}
  By Proposition~\ref{prop:measurability} the function $|\nabla_g^+f|$ is universally
  measurable, and it is easily seen that this implies the
  $\LL^{1}$-measurability of $|\nabla_g^+f|\circ\gamma$ (just
  consider the push forward under $\LL^{1}$ of $\gamma$), see also
  \cite[Lemma~2.4]{AGS2011}.
  Take $\gamma\in \Geo(X)$. The function $f \circ \gamma\colon [0,1] \to \R$
  is Lipschitz and therefore differentiable $\LL^1$-almost everywhere. In particular,
  $|(f \circ \gamma)'(t)| \le l(\gamma)|\nabla_g^+f|(\gamma_t)$ holds and both sides of the inequality are
  well defined at $\LL^1$-almost every point $t \in [0,1]$. Thus
  \[
   |f(\gamma_0)-f(\gamma_1)| = \left|\int_0^1 (f \circ \gamma)'(s)ds \right| \le \int_0^1 |(f \circ \gamma)'(s)|ds
                             \leq  l(\gamma)\int_0^1|\nabla_g^+f|(\gamma_s)ds.
  \]
 \end{proof}

 It is interesting to compare the ascending slope and the upper gradient defined along geodesics to
 the more commonly used versions. First of all, it is immediate that we always have
 \begin{equation}\label{eq:compareslopes}
  |\nabla_g^+f|(x)\le |\nabla_c^+f|(x) \le |\nabla^+f|(x),
 \end{equation}
 where the usual \emph{ascending slope} $|\nabla^+f|(x)$ of $f$ at a point $x$ is defined as
 \[
  |\nabla^+f|(x) = \limsup_{y \to x} \frac{[f(y)-f(x)]^+}{d(y,x)}
 \]
 and the \emph{lower ascending slope along curves} as
 \[
  |\nabla_c^+f|(x) = \sup_{\gamma}\liminf_{s \downarrow 0} \frac{[f(\gamma_s)-f(x)]^+}{d(\gamma_s,x)}
 \]
 with the supremum taken over all curves (recall that by convention all curves we consider are absolutely
 continuous). Moreover, the inequalities in \eqref{eq:compareslopes}
 can be strict. Notice also that the choice of the lower concept (i.e. with the $\liminf_s$)
is justifed by the fact that the upper concept is easily seen to coincide with $|\nabla^+f|$).

 Recall that we have the usual notion of an \emph{upper gradient} $g \colon X \to [0,\infty]$ of a function
 $f \colon X \to \R$
 if we require the inequality \eqref{eq:sugdefinition} to hold along all curves on $[0,1]$, where this time
 $\int_\gamma g$ is understood as
 $\int_0^1g(\gamma_s)|\dot\gamma_s|ds$.
 It is not difficult to show, following the same proof given in
 Proposition~\ref{prop:measurability}, that ascending slopes along curves are
 universally measurable. Moreover, as in Proposition~\ref{prop:lipgrad}, one can prove
 that ascending slopes along curves are upper gradients for Lipschitz functions.

 Ascending slopes along geodesics could be thought, identifying
 geodesics to the tangent space as in the theory of Alexandrov
 spaces, as directional one-sided derivatives. Hence,
 it is natural to ask if ascending slopes along geodesics are also upper gradients (in the usual sense)
 for Lipschitz functions. This is not true in general, as we will see in the next example.

 \begin{example}\label{ex:nongradient}
  There exist a separable complete geodesic metric space $(X,d)$ and a Lipschitz function
  $f \colon X \to \R$ so that $|\nabla_g^+f|$ is not an upper gradient of $f$.

  Let us first construct the metric space $(X,d)$. We start the construction by taking a unit line-segment,
  which we simply denote by $[0,1]$. Next for all $n \in \N$ and $0 \le k < 2^n$ we connect the points $k2^{-n}$ and
  $(k+1)2^{-n}$ in $[0,1]$  with an arc $A_{n,k}$ of length $(2-2^{-n})2^{-n}$. In Figure~\ref{fig:example} the arcs
  $A_{n,k}$ are drawn as half-circles. The space $X$ is then the disjoint union of the arcs $A_{n,k}$ and the initial
  line-segment $[0,1]$.

  We define the distance $d$ between two points $x,y \in X$ as
  \[
   d(x,y) = \inf \sum_{i}l(E_i),
  \]
  where the infimum is taken over all collections of $E_i$'s that connect the points $x$ and $y$,
  $E_i$ are subsets of the arcs and $l(E_i)$ is the length of the piece determined by the length of the arc.
  This way on each arc $A_{n,k}$ the distance is given by the natural distance determined by the length of the arc.
  See the left part of Figure \ref{fig:example} for an illustration of the space.

  Let us check that $(X,d)$ is geodesic. Let $x,\,y \in X$ be two distinct points. If it happens
  that $x$ and $y$ lie on the same arc then the segment of the arc joining the points is our geodesic. We may
  then assume that the points are not on the same arc. We may also assume that $x,y\in [0,1]$. If this is not the case,
  for example $x \notin [0,1]$, we simply notice that any curve connecting the points $x$ and $y$ must go via one of the end-points
  $x'$ and $x''$ of the arc in which $x$ lies in and that
  \[
  d(x,y) = \min\{d(x,x')+d(x',y), d(x,x'')+d(x'',y)\}.
  \]

  We can now find the geodesic between the points $x$ and $y$ with the following procedure. Let $(\gamma^i)_{i=0}^\infty$ be
  a sequence of curves joining $x$ to $y$ so that $\lim_{i \to \infty} l(\gamma^i) = d(x,y)$. Because the lengths of the arcs
  are chosen so that the shortest curve between points $k2^{-n}$ and $(k+1)2^{-n}$ is the arc $A_{n,k}$,
  there exists $i_0 \in \N$ so that each $\gamma^i$, $i \ge i_0$, contains some $A_{n,k}$ with
  \begin{equation}\label{eq:nbound}
   n \le \left\lfloor\frac{\log |x-y|}{\log 2} \right\rfloor+1
  \end{equation}
  and with some $k$. So, taking a subsequence of
  $(\gamma_i)$ we may assume that all the curves contain the same arc $A_{n,k}$. Continuing inductively in
  the same way with the end-points of this arc and the points $x$ and $y$, and finally using a diagonal argument,
  we obtain the geodesic.

  We define the Lipschitz function $f \colon X \to \R$ first on $[0,1]$ by setting $f|_{[0,1]}(x) = x$.
  This fixes the function on the end-points of all the arcs. We continue it inside the arcs by defining for all
  $n \in \N$ and $0 \le k < 2^n$
  \[
   f(\gamma_t) = \begin{cases}
                  k2^{-n} - 2^{-n+1}t, & \text{for }t \in [0,1/2]; \\
                  (k-3)2^{-n} + 2^{-n+2}t, & \text{for }t \in [1/2,1],
                 \end{cases}
  \]
  where $\gamma \colon [0,1] \to X$ is the constant speed geodesic joining $k2^{-n}$ to $(k+1)2^{-n}$ in $A_{n,k}$.
  See the right part of Figure~\ref{fig:example} for the graph of the function along a couple of the arcs and the line $[0,1]$.

  \begin{figure}
   \centering
   \includegraphics[width=0.9\textwidth]{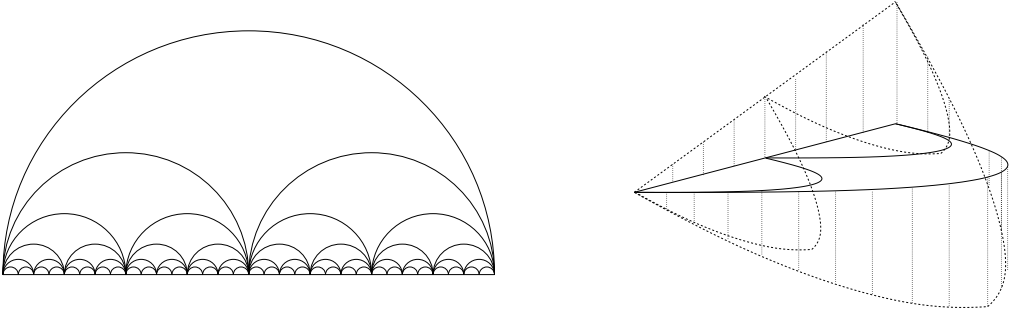}
   \caption{On the left is an illustration of the space $(X,d)$. The lengths of the curves are chosen so
            that the geodesics between two points on the interval prefer going along the longest curves.
            On the right is the graph of the Lipschitz function $f$ drawn along a part of the interval
            and along a couple of the construction curves joining the points of the interval.}
   \label{fig:example}
  \end{figure}

  Let us now show that $|\nabla_g^+f|(x) = 0$ for all $x \in [0,1]$. To see this take a geodesic $\gamma$
  starting from a point $x \in [0,1]$.
  If $\gamma$ near the point $x$ consists only of one piece of an arc the equality
  \[
   \liminf_{s \downarrow 0} \frac{[f(\gamma_s)-f(x)]^+}{d(\gamma_s,x)} = 0
  \]
  is immediate. Suppose then that for every $\epsilon > 0$ there exists a point $y \in [0,1]$ which is in the considered
  geodesic $\gamma$ and $0 < d(x,y) < \epsilon$. As we have noted before, the part of $\gamma$ that connects $x$ to $y$ must contain an
  arc $A_{n,k}$ with $n$ bounded as in \eqref{eq:nbound} and with some $k$. Moreover, we can take such an
  $A_{n,k}$ that $x > (k-1)2^{-n}$. Let $z$ be the middle point of $A_{n,k}$. Now $f(z) = (k-1)2^{-n} < f(x)$,
  and so indeed $|\nabla_g^+f|(x) = 0$.

  On the other hand, the constant speed curve $\gamma \colon [0,1] \to [0,1]$
  has length $2$ and is therefore an admissible test curve for the upper gradient property.
 \end{example}

 It is easy to see that the space $X$ in our previous example is not doubling and that it is extremely branching.
 In light of the example one could still hope many natural conjectures to be true.

 \begin{question}
  Which assumptions are needed on the metric space $(X,d)$ and on the measure $m$ to ensure for any Lipschitz
  function $f \colon X \to \R$ that
  \begin{enumerate}
   \item[(i)] the equality $|\nabla_g^+f|(x) = |\nabla_c^+f|(x)$ holds at $m$-almost every point $x \in X$?
   \item[(ii)] the function $|\nabla_g^+f|(x)$ is an upper gradient of $f$?
  \end{enumerate}
 \end{question}

Notice that if we assume that if $m$ is a doubling measure on $(X,d)$ 
then $|\nabla^\pm f|=|\nabla f|$ $m$-a.e.
in $X$ (see \cite[Proposition~2.5]{AGS2011} for the simple proof of this fact) for any Lipschitz function $f$. 

In addition, if $(X,d,m)$ supports a  $(1,1)$-Poincar\'e inequality, then we can apply Cheeger's theory to obtain 
$|\nabla f|\leq g$ $m$-a.e. in $X$ for any (weak) upper gradient of $f$. Choosing $g=|\nabla^+_cf|$ yields
$$
|\nabla^+_c f|=|\nabla^+f|=|\nabla f|\qquad\text{$m$-a.e. in $X$.}
$$
For the same reason, under the same doubling and Poincar\'e assumptions,
a positive answer to question (ii) implies a positive answer to question
(i): indeed, choosing $g=|\nabla^+_g f|$ one obtains 
$$
|\nabla_g^+f|=|\nabla^+_c f|=|\nabla^+f|=|\nabla f|\qquad\text{$m$-a.e. in $X$.}
$$


 \section{Mass transportation in metric spaces}\label{sec:transport}

 Before stating and proving our main results we briefly discuss in the next subsection the basic properties of
 transport plans
 and maps in the general metric space setting. A comprehensive treatment of the theory can be found for example in
 \cite{AGS2008} and \cite{V2009}.

 \subsection{Basic properties of transport plans}

 Let $\mathscr{P}(X)$ denote the set of all Borel probability measures on $X$.
 The Wasserstein distance between two measures $\mu,\,\nu \in \mathscr{P}(X)$ is defined as
 \begin{equation}\label{eq:definplan}
   W_2(\mu,\nu) = \left(\inf_{\gamma}\int_{X\times X}d^2(x,y)d\gamma(x,y)\right)^{1/2},
 \end{equation}
 where the infimum is taken over all \emph{transport plans} $\gamma$ between $\mu$ and $\nu$, i.e.
 measures $\gamma \in \mathscr{P}(X \times X)$ for which $\texttt{p}_\#^1\gamma = \mu$ and
 $\texttt{p}_\#^2\gamma = \nu$.
 Here the mappings $\texttt{p}^1$ and $\texttt{p}^2$ denote the projections to the first and second coordinate
 respectively.
 The notation $f_\#\mu$ for a measure $\mu \in \mathscr{P}(X)$ and a $\mu$-measurable
 mapping $f\colon X \to Y$ means the \emph{push-forward} measure defined as $f_\#\mu(A) = \mu(f^{-1}(A))$
 for all $A \in \mathscr{B}(Y)$. Notice that in general $W_2(\mu,\nu)$ might be infinite.
 We call a transport plan $\gamma_0$ between two measures $\mu,\,\nu \in \mathscr{P}(X)$,
 for which $W_2(\mu,\nu)<\infty$, \emph{optimal} if the infimum in \eqref{eq:definplan}
 is attained at $\gamma=\gamma_0$.

 Since we are dealing with geodesic spaces, we can equivalently consider geodesic transport plans.
 We define the set of \emph{geodesic plans} between $\mu$ and $\nu$ as the set of
 all
 $\pi \in \mathscr{P}(\Geo(X))$ for which $(e_0)_\#\pi = \mu$, $(e_1)_\#\pi =
 \nu$. We say that a geodesic plan is optimal, and write $\pi\in\GeoOpt(\mu,\nu)$, if
 \[
  \int_{\Geo(X)}d^2(\gamma_0,\gamma_1)d\pi(\gamma) = W_2^2(\mu,\nu)<\infty.
 \]
 Given an optimal geodesic plan $\pi \in \GeoOpt(\mu,\nu)$, it is clear that $(e_0,e_1)_\#\pi$ is an optimal plan.
 Conversely, making a measurable selection of constant speed geodesics
 $\gamma^{xy}$ from $x$ to $y$ and considering the law of
 $(x,y)\mapsto\gamma^{xy}$ under $\gamma$, any optimal plan can be
 ``lifted'' to an optimal geodesic plan with the same cost.

 The Kantorovich formulation of the transportation problem has also a very useful dual formulation: The minimum
 in \eqref{eq:Kantorovich} is equal to
 \[
  2\sup\left\{\int_X\varphi(x)d\mu(x) + \int_X\psi(y)d\nu(y)\right\},
 \]
 where the supremum is taken among all pairs $(\varphi,\psi) \in C_b^0(X) \times
 C_b^0(X)$ satisfying
 $\varphi(x) + \psi(y) \le \tfrac12 d^2(x,y)$.

 We define the \emph{$c$-transform} of a function $\varphi \colon X \to \R\cup\{-\infty\}$ as
 \[
  \varphi^c(x) = \inf_{y \in X}\left\{\frac{d^2(x,y)}{2} - \varphi(y)\right\}.
 \]
 A function $\psi$ is called \emph{$c$-concave} if $\psi = \varphi^c$ for some function $\varphi$.
 This terminology ($c$-transform, $c$-concavity) refers to a general cost function $c$. Here and in the sequel
 the cost function $c$ is given by the halved square of the distance.

 \begin{definition}
  Given an optimal geodesic plan $\pi \in \GeoOpt(\mu,\nu)$ we call a Borel function
  $\varphi \colon X \to \R\cup\{-\infty\}$
  a \emph{Kantorovich potential} (relative to the optimal geodesic plan $\pi$) if it is $c$-concave and
  \[
   \varphi(\gamma_0) + \varphi^c(\gamma_1) = \frac{d^2(\gamma_0,\gamma_1)}{2}
   \qquad\text{for $\pi$-a.e. }\gamma \in \Geo(X).
  \]
 \end{definition}

 Notice that because the Kantorovich potential $\varphi$ is $c$-concave we have $\varphi =
 (\varphi^c)^c$ and that we make no integrability assumption on
 $\varphi$ and $\varphi^c$.

 A set $\Gamma \subset X \times X$ is called \emph{cyclically monotone} if
 \[
  \sum_{i=1}^nd^2(x_i,y_i) \le \sum_{i=1}^nd^2(x_i,y_{\sigma(i)})
 \]
 for any $(x_1,y_1), \ldots, (x_n,y_n) \in \Gamma$ and permutation $\sigma$ of $\{1, \ldots, n\}$.

 Suppose that $W_2(\mu,\nu) < \infty$ and $\pi \in \GeoOpt(\mu,\nu)$. Then $(e_0,e_1)_\#\pi$ is supported on
 a cyclically monotone set and a Kantorovich potential relative to $\pi$ exists.

 \subsection{Brenier theorem in metric spaces}

 As a starting point we prove the following result which originates from \cite[Theorem 10.3]{AGS2011}. The difference
 compared to the original result is in the definition of the ascending slope, here replaced by the lower ascending
 slope. Also, we like to repeat the proof in our situation for the convenience of the reader.

 \begin{proposition}\label{prop:ags1}
  Suppose that $m$ is a finite measure on a bounded space $(X,d)$ and $\mu = \rho m \in \mathscr{P}(X)$, $\nu \in \mathscr{P}(X)$.
  Take $\pi \in \GeoOpt(\mu,\nu)$ and let $\varphi \colon X \to \R \cup \{-\infty\}$ be a Kantorovich potential relative to $\pi$.
  If there exists $\bar s\in (0,1)$ satisfying $(e_s)_\#\pi = \rho_sm$ for all $s\in (0,\bar s)$ and
  \begin{equation}\label{eq:energyestim}
   \limsup_{s \downarrow 0}\int_X \rho_s\log\rho_s dm < \infty,
  \end{equation}
  then
  \[
   |\nabla_g^+\varphi|(\gamma_0) = d(\gamma_0,\gamma_1)\qquad\text{for $\pi$-a.e. $\gamma \in
   \Geo(X)$.}
  \]
 \end{proposition}

 \begin{proof}
 By the definition of the Kantorovich potential we have
 \begin{equation}\label{eq:fromKantdef}
  \varphi(\gamma_0) = \frac{d^2(\gamma_0,\gamma_1)}{2} - \varphi^c(\gamma_1)
 \end{equation}
 for $\pi$-a.e. $\gamma \in \Geo(X)$. On the other hand, for any $z \in X$ we have
 \begin{equation}\label{eq:fromtrans}
  \varphi(z) \le \frac{d^2(z,\gamma_1)}{2} - \varphi^c(\gamma_1).
 \end{equation}
 Thus combining these two we get that for $\pi$-a.e. $\gamma$ it
 holds
 \begin{align*}
  |\nabla_g^+\varphi|(\gamma_0) & \le \limsup_{z \to \gamma_0} \frac{[\varphi(z)-\varphi(\gamma_0)]^+}{d(z,\gamma_0)}
   \le \limsup_{z \to \gamma_0} \frac{[d^2(z,\gamma_1) - d^2(\gamma_0,\gamma_1)]^+}{2d(z,\gamma_0)} \\
  & \le \limsup_{z \to \gamma_0} \frac{d^2(z,\gamma_0) + 2d(\gamma_0,\gamma_1)d(z,\gamma_0)}{2d(z,\gamma_0)} = d(\gamma_0,\gamma_1).
 \end{align*}

 Let us now prove the converse inequality in an integral form.
 Taking $z = \gamma_t$ in \eqref{eq:fromtrans} and combining it with \eqref{eq:fromKantdef} we obtain
 \begin{equation}\label{eq:slope1st}
  \varphi(\gamma_0) - \varphi(\gamma_t) \ge \frac{d^2(\gamma_0,\gamma_1)}{2} - \frac{d^2(\gamma_t,\gamma_1)}{2}
   = \frac{2t-t^2}{2}d^2(\gamma_0,\gamma_1).
 \end{equation}
 Because $X$ is bounded, $\varphi$ is Lipschitz and so by Proposition~\ref{prop:lipgrad} the function $|\nabla_g^+\varphi|$ is
 an upper gradient of $\varphi$ along geodesics. Thus we have
 \[
  (\varphi(\gamma_0)-\varphi(\gamma_t))^2 \le \left(\int_0^t|\nabla_g^+\varphi|(\gamma_s)d(\gamma_0,\gamma_1)ds\right)^2
   \le td^2(\gamma_0,\gamma_1)\int_0^t|\nabla_g^+\varphi|^2(\gamma_s)ds.
 \]
 Dividing this by $d^2(\gamma_0,\gamma_t)$ and integrating over $\Geo(X)$ yields
 \begin{align*}
  \frac{1}{t}\int_0^t& \int_{X}|\nabla_g^+\varphi|^2(x)\rho_sdm(x) ds
   = \frac{1}{t}\int_0^t\int_{\Geo(X)}|\nabla_g^+\varphi|^2(\gamma_s)d\pi(\gamma)ds \\
    & \ge \int_{\Geo(X)}\left(\frac{\varphi(\gamma_0)-\varphi(\gamma_t)}{d(\gamma_0,\gamma_t)} \right)^2d\pi(\gamma)
     \ge \frac{2-t}{2}\int_{\Geo(X)}d^2(\gamma_0,\gamma_1)d\pi(\gamma)
 \end{align*}

  From the assumption \eqref{eq:energyestim} we know that
  \[
   \int_X g\rho_sdm \to \int_Xg\rho dm \quad \text{ as } s \downarrow 0 \text{ for all }g \in L^1(X,m).
  \]
  Because $m$ is finite and $|\nabla_g^+\varphi|$ is bounded,
  this holds also for $g = |\nabla_g^+\varphi|^2$. Therefore
  \begin{align*}
   \int_{X} |\nabla_g^+\varphi|^2d\mu & = \lim_{t \downarrow 0} \frac{1}{t}\int_0^t \int_{X}|\nabla_g^+\varphi|^2(x)\rho_sdm(x) ds\\
    &\ge \lim_{t \downarrow 0}\frac{2-t}{2}\int_{\Geo(X)}d^2(\gamma_0,\gamma_1)d\pi(\gamma).
  \end{align*}
 \end{proof}

 With the help of the Proposition~\ref{prop:ags1} we are now able to prove a Brenier-type theorem
 in strongly non-branching
 metric spaces.

 \begin{theorem}\label{thm:BrenierNonBranch}
  Assume that $(X,d)$ is a strongly non-branching geodesic metric space
  equipped with a doubling measure $m$, and that $\mu = \rho m \in \mathscr{P}(X)$, $\nu \in \mathscr{P}(X)$
  satisfy
  $W_2(\mu,\nu) < \infty$. Assume also that:
  \begin{itemize}
  \item[(a)] for $m$-almost every point $x\in X$ the space $X$ has a non-branching tangent at $x$;
  \item[(b)] there exists a transport plan $\pi \in \GeoOpt(\mu,\nu)$ such that
  for all $s\in [0,1)$ sufficiently small we have $(e_s)_\#\pi\ll m$
  and the densities $\rho_s$ satisfy
  \[
   \limsup_{s \downarrow 0}\int_X \rho_s\log\rho_s dm < \infty.
  \]
  \end{itemize}
  Then the optimal geodesic plan $\pi$ is given by a mapping $T \colon X \to X$,
   i.e. $\gamma_1 = T(\gamma_0)$ for $\pi$-a.e. $\gamma \in \Geo(X)$.
 \end{theorem}

 \begin{proof}
 Suppose that there is no such $T$. If we then fix a point $x_0 \in X$ and consider the restricted and rescaled measures
 \[
  \pi_r = \frac{\pi|_{A(r)}}{\pi(A(r))}, \qquad \text{where }A(r) = \{\gamma \in \Geo(X), \gamma_t \in B(x_0,r) \text{ for all }t \in [0,1]\},
 \]
 we notice that for large enough $r > 0$ the assumptions of the theorem are satisfied and still there exists no such $T$.
 Therefore we may assume the space $X$ to be bounded.

 Let $\varphi$ be the Kantorovich potential relative to $\pi$.
 Let $x\in X$ be a point where the space $X$ has a non-branching tangent. Suppose that there are two geodesics
 $\gamma, \tilde\gamma \in \Geo(X)$ so that $\gamma_0=\tilde\gamma_0 = x$,
 \begin{equation}\label{eq:geodslope}
  |\nabla_g^+\varphi|(x) = d(x,\gamma_1) = d(x,\tilde\gamma_1)
 \end{equation}
 and
 \[
  \varphi^c(\tilde\gamma_1) = \varphi^c(\gamma_1) = \frac{d^2(x,\gamma_1)}{2} - \varphi(x).
 \]

  Let $(Y,d_Y)$ be a non-branching geodesic metric space tangent to $X$ at $x$ and $r_n \downarrow 0$ a sequence so that
  \[
    (X,d_{r_n},x) \rightarrow (Y,d_Y,0)
  \]
  in the pointed Gromov-Hausdorff convergence as $n \to \infty$. Since we assumed
  $(X,d,m)$ to be doubling, the spaces
  $(\overline{B}_{(X,d_{r_n})}(x,1),d_{r_n})$ are easily seen to be
  equi-compact. Indeed, for any $\epsilon>0$ we can find a maximal
  disjoint family of balls of radius $r_n\epsilon/2$ contained in
  $\overline{B}(x,r_n)$, so that the family of balls with doubled
  radius covers $\overline{B}(x,1)$, and use the doubling inequality
  $$
  \mu(B_r(y))\geq C\biggl(\frac r R\biggr)^\alpha\mu(B_R(x))
  \qquad\text{whenever $B_r(y)\subset B_R(x)$}
  $$
  (here $C>0$ and $\alpha>0$ depend on the doubling constant only)
  with $r=r_n\epsilon/2$ and $R=r_n$ to estimate the number of these balls
  with a constant depending only on $C$, $\alpha$ and $\epsilon$.
  We can then apply Theorem~\ref{thm:Gromov} to obtain a compact space $(Z,d_Z)$ and isometric embeddings
  \[
   i_n \colon (\overline{B}_{(X,d_{r_n})}(x,1),d_{r_n}) \to (Z,d_Z),
   \qquad i \colon (\overline{B}_{(Y,d_Y)}(0,1),d_Y) \to (Z,d_Z)
  \]
  so that $i_n(\overline{B}_{(X,d_{r_n})}(x,1))\to i(\overline{B}_{(Y,d_Y)}(0,1))$ in the Hausdorff convergence.

  From \eqref{eq:geodslope} we get for every $n \in \N$ a constant speed geodesic
  $\gamma^n$ with $\gamma_0^n = x$ and a radius $R_n>0$ so that
  \[
   \frac{(\varphi(\gamma_s^n)-\varphi(x))^+}{d(\gamma_s^n,x)} > |\nabla_g^+\varphi|(x)-\frac{1}{n}
  \]
  for every $s \in (0,1)$ for which $d(\gamma_s^n,x) < R_n$.

  Now, possibly taking a subsequence of $(r_n)_{n=1}^\infty$ so that $r_n<R_n$, we get a sequence of points
  $(y_n)_{n=1}^\infty \subset X$ with $d(y_n,x)=r_n$ and
  \begin{equation}\label{eq:svito3}
   \frac{(\varphi(y_n)-\varphi(x))^+}{d(y_n,x)} > |\nabla_g^+\varphi|(x)-\frac{1}{n} = d(x,\gamma_1) - \frac{1}{n}.
  \end{equation}

  Notice also that for any $y \in X$ we have
  \begin{equation}\label{eq:svito2}
   \varphi(y) \le \frac{d^2(y,\gamma_1)}{2} - \varphi^c(\gamma_1) =
   \frac{d^2(y,\gamma_1)-d^2(x,\gamma_1)}{2} + \varphi(x).
  \end{equation}
  Writing $z_n = \gamma_s$ for the $s$ for which $d(\gamma_s,x)=r_n$, triangle inequality and
  geodesic property yield
  \begin{equation}\label{eq:svito1}
  d(y_n,\gamma_1)-d(y_n,z_n)\leq
  d(z_n,\gamma_1)=d(x,\gamma_1)-d(x,z_n).
  \end{equation}
  Hence, using first \eqref{eq:svito1}, then \eqref{eq:svito2} and eventually \eqref{eq:svito3} we have
  \begin{align*}
   0 & \le d(y_n,x) + d(x,z_n) - d(y_n,z_n) \le d(y_n,x) + d(x,\gamma_1) - d(y_n,\gamma_1)\\
     & = d(y_n,x) + \frac{d^2(x,\gamma_1)-d^2(y_n,\gamma_1)}{d(x,\gamma_1)+d(y_n,\gamma_1)} \\
     & = d(y_n,x) + \frac{1}{2}\frac{d^2(x,\gamma_1)-d^2(y_n,\gamma_1)}{d(y_n,x)} \frac{2d(y_n,x)}{d(x,\gamma_1)+d(y_n,\gamma_1)}\\
     & \le d(y_n,x) - \frac{\varphi(y_n)-\varphi(x)}{d(y_n,x)} \frac{2d(y_n,x)}{d(x,\gamma_1)+d(y_n,\gamma_1)} \\
     & < \left(1 -  \frac{2(d(x,\gamma_1)-\frac{1}{n})}{d(x,\gamma_1)+d(y_n,\gamma_1)}\right)d(y_n,x)=o(r_n).
  \end{align*}
  By taking a subsequence we find points $y,\,z \in Y$ so that
  \[
   i_n(y_n) \to i(y) \qquad\text{and}\qquad i_n(z_n) \to i(z).
  \]
  In particular
  \[
   d_Y(y,z) - d_Y(y,0) - d_Y(0,z) = \lim_{n \to \infty} \frac{d(y_n,z_n) - d(y_n,x) - d(x,z_n)}{r_n} = 0
  \]
  and so $0$ lies on some constant speed geodesic $\eta$ in $Y$ joining $y$ to $z$ (obtained by the concatenation
  of the geodesics joining $y$ to $0$ and $0$ to $z$).

  With a similar argument we can show that $0$ lies on some constant speed geodesic $\tilde\eta$ in $Y$ joining $y$
  to $\tilde{z}$, where $\tilde{z}$
  is obtained as the limit $i_n(\tilde{z}_n) \to i(\tilde{z})$ of the points $\tilde{z}_n$ which are taken from
  the geodesic
  $\tilde\gamma$ so that $d(\tilde{z}_n,x) = r_n$. Note that we might have to go to yet another subsequence
  to achieve the convergence to $\tilde{z}$.

  Because the space $X$ is strongly non-branching we have
  \[
   d_Y(z,\tilde{z}) \ge \liminf_{n \to \infty} \frac{d(z_n,\tilde{z}_n)}{r_n} > 0
  \]
  and so the geodesics $\eta$ and $\tilde\eta$ contradict the assumption that $Y$ is non-branching.

  This means that our assumptions on the geodesics $\gamma$ and $\tilde\gamma$ can not be satisfied. Therefore there exists a set
  $A \subset \Geo(X)$ so that $\pi(\Geo(X) \setminus A)=0$ and for every $x \in X$ there is at most one $\gamma \in A$ with
  $x = \gamma_0$. Using the set $A$ we can define the transport map $T$ as
  \[
   T(x) = \begin{cases}
           \gamma_1, & \text{when }x = \gamma_0 \text{ for some }\gamma \in A \\
           x, & \text{otherwise.}
          \end{cases}
  \]
 \end{proof}

 \begin{remark}
  One could prove Theorem \ref{thm:BrenierNonBranch} also under slightly different assumptions. Namely by weakening the
  definition of strongly non-branching metric space by replacing the liminf in \eqref{eq:strongnonbrancdef} to limsup,
  and then assuming that at almost every point all the tangent spaces to $X$ are non-branching. This modified theorem is achieved
  by letting the sequence of radii in the blow-up be dictated by the weakened form of strong non-branching
  property, namely choosing $r_n$ in such a way that
  $d(x,z_n)=d(x,\tilde{z}_n)=r_n$ and
  $\lim_nd(z_n,\tilde{z}_n)/r_n>0$.
 \end{remark}

 Theorem \ref{thm:BrenierNonBranch} applies, for example, when the space $(X,d)$ is a finite dimensional Alexandrov space and $m$
 is the corresponding volume measure on $X$. The estimate \eqref{eq:energyestim} on the relative entropy follows in this case
 from the result of Petrunin \cite{P2010} which shows that in Alexandrov spaces the functional
 \begin{equation}\label{eq:renyientropy}
    \mu \mapsto \int_X \rho^{1-\frac{1}{N}}dm
 \end{equation}
 is concave along Wasserstein geodesics. Notice that a different proof for the Brenier theorem in Alexandrov
 spaces was already given by Bertrand in \cite{B2008}.

 Brenier theorem has been recently established by Gigli \cite{G2011} in non-branching spaces with Ricci-curvature bounded from below.
 This generalizes the previous result by Bertrand and it covers for example the case where the functional \eqref{eq:renyientropy}
 is concave along geodesics in the Wasserstein space $(\mathcal{P}(X),W_2)$ of a non-branching space $(X,d)$. Whereas our proof
 of Theorem \ref{thm:BrenierNonBranch} is based on the behaviour of blow-ups and the Kantorovich potential, the proof by Gigli
 relies on the concavity of the functional and does not use the Kantorovich potential at all. Notice that because of the different
 techniques used in the proofs our geometric assumptions on the metric space $X$ differ from those assumed by Gigli and hence the
 two theorems cover different collection of metric spaces.

 It is also important to notice that our Theorem \ref{thm:BrenierNonBranch} by no means covers all the cases where
 the Brenier theorem is known to hold. For example the Brenier theorem holds in the Heisenberg group \cite{AR2008}, but
 it is not difficult to see that the Heisenberg group is not strongly non-branching.

 We end this paper with an improvement of \cite[Theorem 10.4.]{AGS2011} in the case where the reference measure $m$ is doubling.
 In \cite{AGS2011} it was shown that without the assumption that $m$ is doubling we have the conclusion
 \[
   \lim_{t \downarrow 0}\frac{\varphi(\gamma_0)-\varphi(\gamma_t)}{d(\gamma_0,\gamma_t)} = d(\gamma_0,\gamma_1)
             \qquad \text{ in }L^2(\Geo(X),\pi)
 \]
 in the following theorem. Because our Proposition \ref{prop:ags1} was proved in the case where the space $X$ is
 bounded, we
 will make the same boundedness assumption here. As in many of the results in \cite{AGS2011} we could remove
 this assumption by requiring
 the density of the initial measure $\mu$ with respect to $m$ to be uniformly bounded away from zero.

 \begin{theorem}\label{thm:BrenierImprovement}
  Let $m$ be a doubling measure on a bounded metric space $X$ and $\mu = \rho m \in \mathscr{P}(X)$,
  with $\rho>0$ $m$-a.e. in $X$ and $\nu \in \mathscr{P}(X)$. Let
  $\pi\in \GeoOpt(\mu,\nu)$ and $\varphi \colon X \to \R \cup\{-\infty\}$ be a
  Kantorovich potential relative to $\pi$ satisfying
  \begin{equation}\label{eq:newhypo}
  |\nabla_g^+\varphi|(\gamma_0)=d(\gamma_0,\gamma_1)
  \qquad\text{$\pi$-a.e. in $\Geo(X)$.}
  \end{equation}
  Further assume that there exists $\bar s\in (0,1]$ such that
  for all $s \in [0,\bar s)$ we have $(e_s)_\#\pi\ll m$.
  Then for $\pi$-a.e. $\gamma \in \Geo(X)$ we have
  \[
   \lim_{t \downarrow 0}\frac{\varphi(\gamma_0)-\varphi(\gamma_t)}{d(\gamma_0,\gamma_t)} = d(\gamma_0,\gamma_1).
  \]
 \end{theorem}
 \begin{proof} If the lower ascending slope along geodesics of the
 Kantorovich potential were continuous, the theorem would follow
 immediately from the fact that the lower ascending slope is an
 upper gradient. This is not true in general, but what we can prove
 using density points and cyclical monotonicity is that for
 $\pi$-almost every geodesic the lower ascending slope is continuous
 along the geodesic at its starting point.

  As we have seen in the proof of Proposition \ref{prop:ags1}, in \eqref{eq:slope1st}, the inequality
  \[
   \liminf_{t \downarrow 0}\frac{\varphi(\gamma_0)-\varphi(\gamma_t)}{d(\gamma_0,\gamma_t)} \ge d(\gamma_0,\gamma_1)
  \]
  holds for $\pi$-a.e. $\gamma \in \Geo(X)$.

  On the other hand, by Proposition~\ref{prop:lipgrad} we know that $|\nabla_g^+\varphi|$ is an upper gradient of
  $\varphi$ along geodesics and so for all $\gamma \in \Geo(X)$ the estimate
  \[
   \frac{\varphi(\gamma_0)-\varphi(\gamma_t)}{d(\gamma_0,\gamma_t)}
    \le \frac{1}{d(\gamma_0,\gamma_t)} \int_{\gamma|_{[0,t]}}|\nabla_g^+\varphi|
  \]
  holds for all $t \in (0,1)$. So, our claim follows if we can show with any $\delta > 0$ that for
  $\pi$-a.e. $\gamma \in \Geo(X)$
  \begin{equation}\label{eq:aes}
    |\nabla_g\varphi^+|(\gamma_s) \le (1+\delta)d(\gamma_0,\gamma_1) \qquad \text{for } \LL^1\text{-a.e. }s \in (0,t)
  \end{equation}
  when $t>0$ (depending on $\delta$ and $\gamma$) is small enough.

  Because $\rho>0$, we know from \eqref{eq:newhypo} that for $m$-a.e. $x \in X$ there exists
  $\gamma^x \in \Geo(X)$ with $\gamma^x_0 = x$ and
  $|\nabla_g^+\varphi|(x) = d(\gamma^x_0,\gamma^x_1)$.  When we combine this with the assumption $(e_s)_\#\pi \ll m$,
  for any $s \in (0,\bar s)$, we get
  for $\pi$-a.e. $\gamma \in \Geo(X)$ a curve $\hat\gamma \in \Geo(X)$ so that $\gamma_s = \hat\gamma_0$ and
  \begin{equation*}\label{eq:gradagree}
  |\nabla_g^+\varphi|(\hat\gamma_0) = d(\hat\gamma_0,\hat\gamma_1).
  \end{equation*}
  Hence, by Fubini's theorem we know that for $\pi$-a.e. $\gamma \in \Geo(X)$
  a curve $\hat\gamma \in \Geo(X)$ with the above properties exists
  for $\LL^1$-a.e. $s \in (0,\bar s)$.

  Our task is now to estimate $d(\hat\gamma_0,\hat\gamma_1)$ from above. Because $m$ is doubling, it is enough to prove this
  for $\gamma \in \Geo(X)$ for which $\gamma_0$ is a Lebesgue point of $|\nabla_g^+\varphi|$.
  Let $\epsilon > 0$ and take $t \in (0,1)$ so small that
  for every $0 < s < t$ and $x \in B(\gamma_0, r_s)$, where $r_s = 2sd(\gamma_0,\gamma_1)$, there exists
  $\tilde\gamma \in \Geo(X)$ with $\tilde\gamma_0 \in B(x,\epsilon r_s)$ and
  \begin{equation}\label{eq:Lebesguepoint}
  d(\tilde\gamma_0,\tilde\gamma_1)=
   |\nabla_g^+\varphi|(\tilde\gamma_0)<(1+\epsilon)|\nabla_g^+\varphi|(\gamma_0)
   =(1+\epsilon)d(\gamma_0,\gamma_1).
  \end{equation}
  Define $q = s d(\gamma_0,\gamma_1)/d(\hat\gamma_0,\hat\gamma_1)$ and let $x = \hat\gamma_q \in B(\gamma_0, r_s)$.
  With this choice of $x$ let $\tilde\gamma \in \Geo(X)$ be as above. Notice that we may assume $q < s$, as
  otherwise the upper bound on $ d(\tilde\gamma_0,\tilde\gamma_1)$
  immediately follows. The selected curves are illustrated in
  Figure~\ref{fig:cycmon}.

  \begin{figure}
   \psfrag{x}{$\hat\gamma_q$}
   \psfrag{a0}{$\gamma_0$}
   \psfrag{a1}{$\gamma_1$}
   \psfrag{b0}{$\hat\gamma_0 = \gamma_s$}
   \psfrag{b1}{$\hat\gamma_1$}
   \psfrag{c0}{$\tilde\gamma_0$}
   \psfrag{c1}{$\tilde\gamma_1$}
   \psfrag{r}{$r_s$}
   \psfrag{er}{$\epsilon r_s$}
   \centering
   \includegraphics[width=0.9\textwidth]{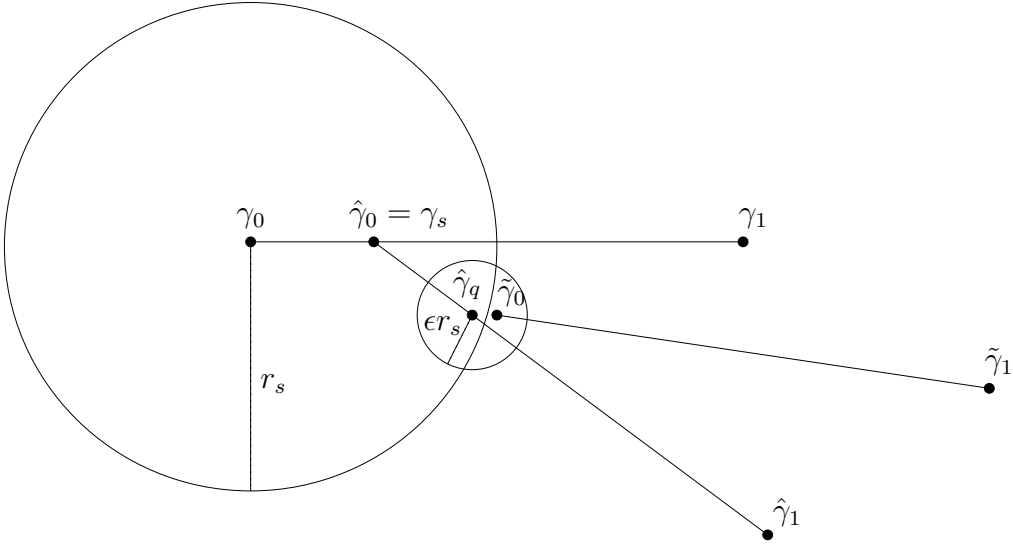}
   \caption{The three curves $\gamma, \hat\gamma$ and $\tilde\gamma$ which are used in the proof.}
   \label{fig:cycmon}
  \end{figure}

  Now we are ready to estimate $d(\hat\gamma_0,\hat\gamma_1)$ from above. For this we use cyclical monotonicity:
  \begin{eqnarray*}
   d^2(\hat\gamma_0,\hat\gamma_1)&+& d^2(\tilde\gamma_0,\tilde\gamma_1)
    \le d^2(\tilde\gamma_0,\hat\gamma_1) + d^2(\hat\gamma_0,\tilde\gamma_1)\\
    &\le& (d(\tilde\gamma_0,\hat\gamma_q) + d(\hat\gamma_q,\hat\gamma_1))^2 +
    (d(\hat\gamma_0,\hat\gamma_q)+d(\hat\gamma_q,\tilde\gamma_0)+d(\tilde\gamma_0,\tilde\gamma_1))^2\\
    &=&  d^2(\tilde\gamma_0,\hat\gamma_q) + 2d(\tilde\gamma_0,\hat\gamma_q)d(\hat\gamma_q,\hat\gamma_1) +
      d^2(\hat\gamma_q,\hat\gamma_1) + d^2(\hat\gamma_0,\hat\gamma_q)
      + 2d(\hat\gamma_0,\hat\gamma_q)d(\hat\gamma_q,\tilde\gamma_0)
      \\&+& 2d(\hat\gamma_0,\hat\gamma_q)d(\tilde\gamma_0,\tilde\gamma_1)
       + d^2(\hat\gamma_q,\tilde\gamma_0) + 2d(\hat\gamma_q,\tilde\gamma_0)d(\tilde\gamma_0,\tilde\gamma_1) +
      d^2(\tilde\gamma_0,\tilde\gamma_1).
   \end{eqnarray*}
    Now, using the inequalities
    $d(\tilde\gamma_0,\hat\gamma_q)<\epsilon r_s$ and
    $d(\gamma_0,\hat\gamma_q)<r_s$ we get that
    $d^2(\hat\gamma_0,\hat\gamma_1)+
    d^2(\tilde\gamma_0,\tilde\gamma_1)$ is bounded from above by
  \begin{eqnarray*}
     & ~\epsilon^2r_s^2 + 2\epsilon r_s(1-q)d(\hat\gamma_0,\hat\gamma_1) + (1-q)^2d^2(\hat\gamma_0,\hat\gamma_1) + q^2d^2(\hat\gamma_0,\hat\gamma_1) \\
      & + 2q\epsilon r_sd(\hat\gamma_0,\hat\gamma_1) + 2qd(\hat\gamma_0,\hat\gamma_1)d(\tilde\gamma_0,\tilde\gamma_1) + \epsilon^2r_s^2 + 2\epsilon r_s d(\tilde\gamma_0,\tilde\gamma_1) + d^2(\tilde\gamma_0,\tilde\gamma_1) \\
    = & ~ 2\epsilon r_s\left(\epsilon r_s + d(\hat\gamma_0,\hat\gamma_1)+ d(\tilde\gamma_0,\tilde\gamma_1)\right) + d^2(\hat\gamma_0,\hat\gamma_1) + d^2(\tilde\gamma_0,\tilde\gamma_1) \\
      & + 2qd(\hat\gamma_0,\hat\gamma_1)d(\tilde\gamma_0,\tilde\gamma_1) + 2(q-1)qd^2(\hat\gamma_0,\hat\gamma_1) \\
    = & ~ 2\epsilon r_s\left(\epsilon r_s + d(\hat\gamma_0,\hat\gamma_1)+ d(\tilde\gamma_0,\tilde\gamma_1)\right) + d^2(\hat\gamma_0,\hat\gamma_1) + d^2(\tilde\gamma_0,\tilde\gamma_1) \\
      & + r_sd(\tilde\gamma_0,\tilde\gamma_1) +
      (q-1)r_sd(\hat\gamma_0,\hat\gamma_1).
  \end{eqnarray*}
  It follows that $2\epsilon r_s\left(\epsilon r_s + d(\hat\gamma_0,\hat\gamma_1)+
  d(\tilde\gamma_0,\tilde\gamma_1)\right)
  + r_sd(\tilde\gamma_0,\tilde\gamma_1) +
      (q-1)r_sd(\hat\gamma_0,\hat\gamma_1)\geq 0$, so that dividing
      by $r_s$ and using \eqref{eq:Lebesguepoint} yields
  \[
   d(\hat\gamma_0,\hat\gamma_1) \le \frac{(1+2\epsilon)d(\tilde\gamma_0,\tilde\gamma_1) + 2\epsilon^2r_s}{1-q-2\epsilon}
    \le \frac{(1+2\epsilon)(1+\epsilon) + 4\epsilon^2s}{1-s-2\epsilon}d(\gamma_0,\gamma_1).
  \]
  Choosing $s$ and $\epsilon$ small enough, depending on $\delta$, we
  achieve \eqref{eq:aes} and conclude the proof.
 \end{proof}

\end{document}